\documentclass[12pt]{amsart}

\usepackage{amssymb}
\usepackage{amsmath}     
\usepackage{graphicx}

\usepackage{xcolor}

\usepackage[all]{xy}

 \usepackage{url}
\usepackage[inline]{enumitem}

 \newlength{\boxwidth}
\theoremstyle{plain}
\newtheorem{thm}{Theorem}[section]
\newtheorem*{MT}{Main Theorem}

\newtheorem{cor}[thm]{Corollary}

\theoremstyle{definition}

\newtheorem*{alg2*}{Main Algorithm 2}
\newtheorem*{alg1*}{Main Algorithm 1}

\newtheorem{rem}[thm]{Remark}

\newtheorem{innercustomalg}{Algorithm}
\newenvironment{customalg}[1]
{\renewcommand\theinnercustomalg{#1}\innercustomalg}
{\endinnercustomalg}

\numberwithin{equation}{section}

\def\RR{{\mathbb R}}
\def\ZZ{{\mathbb Z}}
\def\CC{{\mathbb C}}
\def\QQ{{\mathbb Q}}
\def\HH{{\mathbb H}}

\def\cL{{\mathcal L}}

\DeclareMathOperator{\End}{{End}}
\DeclareMathOperator{\Hom}{{Hom}}

\DeclareMathOperator{\Sp}{{Sp}}
\DeclareMathOperator{\Ima}{{Im}}

\begin{document}

\title[Decomposing abelian varieties: algorithms and applications]{Decomposing abelian varieties into simple factors: algorithms and applications.} 
\author{Rub\'i E. Rodr\'iguez}
\address{Departamento de Matem\'atica y Estad{\'\i}stica, Universidad de La Frontera, Temuco, Chile}
\email{rubi.rodriguez@ufrontera.cl}%

\author{Anita M. Rojas}
\address{Departamento de Matem\'aticas, Facultad de
Ciencias, Universidad de Chile, Santiago\\Chile}
\email{anirojas@uchile.cl }%

\thanks{Partially supported by ANID/CONICYT Fondecyt grants 1230708 and 1230034.}

 \begin{abstract}
 
We give an effective procedure to explicitly find  the decomposition of a polarized abelian variety into its simple factors if a  period matrix is known. Since finding this datum is not easy, we also provide two methods to compute the period matrix for a polarized abelian variety, depending on the given geometric information about it.  

These results work particularly well in combination with our previous work on abelian varieties with group actions, since they allow us to fully decompose such varieties by successively decomposing their factor subvarieties, even when these no longer have  a group action. We highlight that we do not require to determine the full endomorphism algebra of any of the (sub)varieties involved. 

We illustrate the power of our algorithms with two byproducts:  we find a completely decomposable Jacobian variety of dimension $101$, filling this Ekedahl-Serre gap, and  we  describe a new completely decomposable Jacobian variety of CM type of dimension $11$.
\end{abstract}

\maketitle

\section{Introduction}\label{S:Intro}

A period matrix $\Pi=\left( E\; Z\right)$ for a polarized abelian variety $A$, defining the relation between the real and the complex coordinate functions  of its lattice and of its vector space respectively, captures deep geometric information about $A$. For instance, if $A$ is defined over $\overline{\mathbb{Q}}$ and has dimension $g$, then $Z$ is an algebraic point in the Siegel space $\mathbb{H}_g$ if and only if $A$ is of CM type; that is, if and only if the simple factors of $A$ have complex multiplication, see \cite{sw}. As a consequence, period matrices are useful tools to describe loci of moduli spaces of abelian varieties  with interesting geometric or arithmetic properties. 

A criterion in terms of period matrices for a polarized abelian variety to be non-simple is given in \cite[Thm. 4.1]{alrJPAA}. This criterion can be roughly stated as \textit{$A$ is not simple if and only if there is a differential form $\omega \in H^{1,1}(A)$ satisfying certain equations given in terms of the period matrix $\Pi=(E\, Z)$ of $A$}; see Section \ref{S:Beyond} for details.

In this work we further improve this criterion, transforming it into an effective tool to decompose the variety $A$ into a product of subvarieties, and ultimately to find the Poincar\'e decomposition of $A$, in an inductive procedure using the results given here to find the period matrices of abelian varieties. 

Our main result (Theorem \ref{beyond}) in this regard can be summarized as follows.

\begin{MT}\label{T:M}
Let $A$ be a polarized abelian variety of dimension $g$ with period matrix $\Pi = (E\, Z)$. Look for $\omega$ as in \cite[Thm. 4.1]{alrJPAA}. If such an $\omega$ exists, then construct the corresponding subvarieties $A_{\omega}$ and its complement $A_{\omega}^c$ according to Theorem \ref{beyond}, obtaining an isogeny decomposition $A_{\omega}\times A_{\omega}^c\to A$ of $A$.

Next compute the period matrices of $A_{\omega}$ and $A_{\omega}^c$ using Theorem \ref{T:matrizriemannchica}, and apply  the procedure again to both subvarieties. 

This algorithm stops when $A$ is decomposed as a product of simple factors: in its Poincar\'e decomposition.

Notice that from \cite[Thm. 4.1 ]{alrJPAA}, if there is no such differential form $\omega$, then $A$ is simple. 
\end{MT}

In general, it is not easy to find either the simple factors of $A$ or to compute explicitly the Riemann matrix for $A$. Apart from algorithms that  are mostly applicable to the case of Jacobian varieties of special curves, as in \cite{dh, gsst, rr1, rr2, bcgr}, and others that are usually based on a numerical approach for compact Riemann surfaces given as plane algebraic curves over number fields \cite{mascot}. Precise results, like the outputs the algorithms given here produce, have been only given for special families of curves, such as by  Weil, who worked out the case of Lefschetz  surfaces $y^p =  x^a (1-x)$, for $p$ prime and $1 \leq a \leq p-1$,  and by Rohrlich, for the case of Fermat's curves $x^n + y^n = 1$. 

In \cite{brr} we gave the theoretical basis and an algorithm to compute Riemann matrices for Jacobian varieties of compact Riemann surfaces with automorphisms, here we extend that algorithm to the case of polarized abelian varieties with a group action (Theorem \ref{T:achica}).

We recall some motivating questions on the subject that can be tackled by decomposing abelian varieties, or Jacobian varieties in particular. Ekedahl and Serre \cite{ekse} studied completely decomposable Jacobian varieties; that is, Jacobians which are isogenous to a product of elliptic curves. In their theorem, they listed several genera in which there are completelly decomposable Jacobians, the largest being $1297$,  but they left several {\it gaps}. Besides, they asked two questions which remain open: Is it true that for every $g>0$, there is a completely decomposable Jacobian variety of dimension $g$?;  Is there a bound for the genus of a curve with completely decomposable Jacobian?. Currently, the smallest {\it Ekedahl-Serre gap} is $g=38$, according to \cite[3.11]{serre2} and \cite{PA}, and $g=101$ is the first gap after $100$; this gap is filled in Section \ref{S:g101}.

On the other hand, Beauville \cite{beau} points out that few examples of curves with Jacobian variety of maximal Picard number $\rho=h^{1,1}$ are known. We recall \cite[Prop. 3]{beau} this nice characterization of $\rho$-maximal abelian varieties; an abelian variety $X$ of dimension $g$ is $\rho$-maximal if and only if it is isogenous to $E^g$, with $E$ an elliptic curve with complex multiplication. Moreover, this is the case if and only if $X$ is isomorphic to a product of mutually isogenous elliptic curves with complex multiplication, and if and only if the rank of $\text{End}(X)$ over the integers  is $2g^2$. 

Having in mind all these fundamental questions and problems, our original motivation for the work in this manuscript was to compare the Group Algebra Decomposition (GAD)
$$
A \sim B_1^{n_1} \times \ldots \times B_r^{n_r} ,
$$
known to exist for any abelian variety $A$ with the action of a group $G$, 
with the well known Poincar\'e decomposition into simple factors 
$$
A \sim C_1^{k_1} \times \ldots \times C_s^{m_s},
$$
valid for any polarized abelian variety. In general, there is no correspondence between them.

The factors in the Poincar\'e decomposition satisfy $\Hom(C_i,C_j)=0$, whereas in GAD they satisfy  $\Hom_G (B_i^{n_i}, B_j^{n_j}) =0$, for $i\neq j$.  Now we know that it may well happen that two different \textit{primitive factors} $B_i$ and $B_j$ are isogenous, or that a $B_j$ is non-simple, as we will see in the applications in Sections \ref{S:g101} and \ref{S:g11}.

One typical application of our new results is to go beyond the GAD decomposition: given a GAD for $A$, with Theorem \ref{T:achica} first compute the period matrices for the isotypical factors $B_i^{n_i}$. Then, with Theorem \ref{T:matrizriemannchica}, compute the period matrices of the primitive factors $B_i$. Finally, with Theorem \ref{beyond} sketched as Main Theorem above, find the Poincar\'e decomposition of every primitive factor, hence fully decomposing the original variety $A$.

An interesting feature is that, as can be seen from what follows,  a priori knowledge of $\End_{\mathbb{Q}}(A)$ is not required.

In particular, and as a way of illustrating the kind of results that can be found applying the results presented here, we give the following applications.

\begin{cor}\label{P:coro}
	There is a curve of genus $g=101$ with completely decomposable Jacobian variety, thus filling an Ekedahl-Serre gap. 
	
	There is a curve of genus $11$ isogenous to a product of elliptic curves with complex multiplication.
\end{cor}

Our methods can be translated into algorithms, one of which is presented in the Appendix, while the rest are available at \cite{rrgithub}, together with all the codes and the full calculations for the applications.

\medskip

The structure of this work is as follows. After some preliminaries given in Section \ref{S:matricesperiodos}, 
in Section \ref{S:achica} we give a method to compute a period matrix for any $G$-invariant subvariety $B$ of an abelian variety $A$ with  action of a group $G$, given the symplectic representation of  the action of $G$ on $A$. First, we describe how to find the restriction of the action of $G$ from $A$ to $B$, and then we use that action to find the period matrix for $B$. This result is applied to find the period matrix of the isotypical factors; these are the subvarieties of $A$ corresponding to images of central idempotents in $\mathbb{Q}[G]$.

This algorithm can be used for $A$ itself, but the computer runs out of memory very fast as the dimension of $A$ grows. So it is better to use it, as said, in combination with the isotypical decomposition, since the isotypical factors $A_j$ are $G$-invariant subvarieties of $A$. This algorithm is an improvement of the one we developed in \cite{brr} for principally polarized abelian varieties. Here we generalize it to any type of polarization, since the induced polarization on a subvariety is not necessarily principal.

The second method is presented in Section \ref{S:matrizriemannchica}, on how to compute a period matrix for the subvariety $A_f=\Ima(f)$ of  a polarized abelian variety $(A=V/L, \cL)$ of dimension $g$, given a period matrix $\Pi=(E\, \, Z)$ for $A$ (with respect to some bases $\alpha$ and $\beta$ for $V$ and $L$ respectively), and $f\in \End_{\QQ}(A)$ represented in the basis $\beta$ as a matrix in $M_{2g}(\QQ)$. Observe that any subvariety of an abelian variety is the image of $A$ under some endomorphism of $A$ (for instance its norm map), so this is a general procedure.

In Section \ref{S:Beyond}, we recall  from \cite{alrJPAA} a characterization for an abelian variety to be simple in terms of its period matrix.  In subsection \ref{SS:nosotras} we put together all the results mentioned earlier with this criterion to obtain the Main Theorem, there called Theorem \ref{beyond}, which may be thought of as an algorithm to compute the Poincar\'e decomposition of an abelian variety into simple factors given its period matrix. 

Sections \ref{S:g101} and \ref{S:g11} contain the proof of Corollary \ref{P:coro}, as a combined application of all the  results.

In Section \ref{S:app} we outline one of the algorithms emerging from our results; the one for finding the Poincar\'e decomposition of an abelian variety. The code for this algorithm and the others in this work are in  \cite{rrgithub}; the reader can also find there  more precise explanations on how to actually implement them in Magma \cite{magma}, as well as the calculations for our applications.

\section{Preliminaries}\label{S:matricesperiodos}

We recall here some known results about decompositions of abelian varieties with a group action; we refer to  \cite{lre,cr2}, \cite[Ch.13]{bl} and \cite{lrb}  for details. Let $A=V/L$ be an abelian variety with the action of a (finite) group $G$;
this action induces an algebra homomorphism 
$$
\rho:\QQ[G]\to \End_{\QQ}(A).
$$

The semi-simple algebra $\QQ[G]$ decomposes as the product of unique simple algebras $\QQ[G]e_j$, where each $e_j$ is the central idempotent corresponding to the rational irreducible representation $W_j$  of $G$, with $j$ in $1, \ldots , r$ indexing a full set of non-equivalent rational irreducible representations of $G$. This induces the so called 
\textit{isotypical decomposition} of $A$, given by (unique) abelian subvarieties $A_1 , \ldots , A_r$ of $A$, with $G$ acting on $A_j$ by (an appropriate multiple of) the rational irreducible representation  $W_j$ of $G$, and such that the sum morphism is a $G$-equivariant isogeny:  
$$
A_1 \times \ldots \times A_r \rightarrow A. 
$$
 
As each \textit{isotypical factor} $A_j$ is described explicitly as the image of $A$ under $\rho(e_j)$ in $\End_{\QQ}(A)$,  we can compute the period matrix of each $A_j$ if the rational representation of the action of $G$ on $A$ is known, according to our second method described in Theorem \ref{T:achica}, which includes how to find the (restricted) action of $G$ on $A_j$.   If, on the other hand, the period matrix of $A$ is known,  then the period matrix of each $A_j$ may be computed using Theorem \ref{T:matrizriemannchica}. The first approach is more common, since the $G$-abelian subvarieties $A_j$ are lower dimensional than $A$, and hence computations are simpler.

Since each simple algebra $\QQ[G]e_j$ can in turn be decomposed as a sum of primitive left ideals, another decomposition is obtained: a \textit{group algebra decomposition} (GAD) of $A$. Its form is
\begin{equation} \label{e:gad}
B_1^{n_1}\times \dots \times B_r^{n_r}\to A,
\end{equation}
where each $B_j$ is a subvariety of $A_j$ and $n_j=\frac{\dim V_j}{m_j}$, with $m_j$ the Schur index of $V_j$ (any complex irreducible component of $W_j\otimes \mathbb{C}$). The factors $B_j$ are called \textit{primitive factors} in the GAD, since they correspond to images of primitive idempotents in $\QQ[G]$. Note that they are not uniquely defined, and different choices for them correspond to different GAD's of $A$ having, for instance, different isogeny degrees.

The starting point of this work is the method in \cite{lroarch}, which allows the computation of the polarization induced on the isotypical factors for the case when $A$ is the Jacobian variety of a curve with group action, if  the rational representation $\rho_r:G\to \Sp_{2g}(\ZZ)$ of the group $G$ is known. In \cite{lrrMathZ}, the method was extended to compute the induced polarization on any subvariety of a polarized abelian variety $A$ with group action (by a group $G$) given as the image of an element of $\QQ[G]$.

In this work we go further, obtaining  the period matrices for any subvariety of $A$ defined as the image of an element $f\in \End_{\QQ}(A)$. In particular, for the primitive factors $B_j$. We explain this in section \ref{S:matrizriemannchica}.

We first recall some notation and well known facts, as described in \cite{bl}. Let $(A=V/L , J_E)$ be a polarized abelian variety (\textit{pav} in what follows), where $J_E$ denotes the polarization considered as an integral alternating matrix on the lattice $L$.  Let $(d_1, \dots , d_g)$ be the type of
the polarization $J_E$; a \textit{symplectic basis} for this polarization is a basis $\beta$ of $L$ with respect to which the alternating form is given by the matrix
$$
J_E := \left( \begin{array}{cc}
        0 & E \\
       -E & 0
      \end{array} \right)
$$
with $E = \text{diag}(d_1, \dots, d_g)$.

Given a polarized abelian variety $(A=V/L , J_E)$ of dimension $g$ and bases $\alpha = \{v_1 , \ldots , v_g\}$ of $V$ and $\beta =\{ \lambda_1 , \ldots , \lambda_{2g}\}$ of $L$, the  \textit{period matrix} $\Pi =(\Pi_{j,i} )$ of $A$ with respect to these bases is a $g \times 2g$ complex matrix given by the coefficients of $\lambda_i$ expressed in terms of the $v_j$:
$$
\lambda_i = \sum_{j=1}^g \Pi_{j,i} \, v_j \ , \ 1 \leq i \leq 2g.
$$  

If the basis $\beta =\{ \lambda_1 , \ldots , \lambda_{2g}\}$ is chosen as symplectic with $E = \text{diag}(d_1, \dots, d_g)$, then $\alpha = \{v_1 = \frac{1}{d_1} \lambda_1 , \ldots , v_g = \frac{1}{d_g} \lambda_g\}$ is a basis for $V$, and  the period matrix for $A$ with respect to these bases has the form $\Pi = \left( E \ \ Z \right),$  with $Z$ in the Siegel space 
$$
\mathbb{H}_g = \{ Z \in M(g \times g, \mathbb{C}) : \mbox{}^tZ = Z , \Im Z > > 0 \};
$$ 
in this case $Z$ is called a \textit{Riemann matrix} for $A$.

Given two polarized abelian varieties $(A=V/L , J_E)$  and $(A'=V'/L' , J_{E'})$, of respective dimensions $g$ and $g'$, choose bases for $V$, $L$, $V'$ and $L'$, and denote the respective period matrices by $\Pi$ and $\Pi'$. To any homomorphism $f : A \to A'$ one can associate two matrices with respect to the corresponding bases: the analytic representation $\rho_a(f) : V \to V'$ of $f$: a $g' \times g$ complex matrix, and the rational representation $\rho_r(f) : L \to L'$ of $f$: a $2g' \times 2g$ integral matrix. 

The fundamental relation (the \textit{Hurwitz relation}) that connects them is given by
\begin{equation} \label{eq:Hurwitz}
	\rho_a(f) \ \Pi = \Pi' \, \rho_r(f).
\end{equation} 

If $f$ is biholomorphic, then $g'=g$, $\rho_a(f)$ is a nonsingular matrix, and $\rho_r(f)$ is a unimodular matrix. Conversely, if $C$ and $N$ are, respectively,   nonsingular and  unimodular $g \times g$ and $2g \times 2g$ matrices satisfying $C \ \Pi = \Pi' \ N$, then $C$ is the matrix of an invertible linear map $F : V \to V'$ that satisfies $F(L) = L'$ and covers an isomorphism $f : A \to A'$. More generally, an isogeny $f$ corresponds to a nonsingular $\rho_a(f)$; 
the order of the kernel of $f$ is called the \textit{degree of the isogeny}, and it equals  $|\det(\rho_r(f))|$.

Furthermore, if the bases are chosen so that the period matrices have the form  $\Pi = \left( E \ \ Z \right)$ and $\Pi' = \left( E \ \ Z' \right)$, then an isomorphism $f$ preserves the polarization if and only if $\rho_r(f)$ belongs to 
$$
\Sp^E(2g,\mathbb{Z}) =\{ N \in M(2g\times 2g, \mathbb{Z}): N^t \cdot J_E \cdot N =J_E\}
$$ 
with
$ 
J_E := \left( \begin{array}{cc}
	0 & E \\
	-E & 0
\end{array} \right)
$ as before, where $N^t$ denotes the transpose of $N$.

For any $f \in \End_{\QQ}(A)$, the subspace of $V$ and the sublattice of $L$ defining the abelian subvariety $A_f:=\Ima(f)$ will be denoted by $V_f$ and $L_f$ respectively.

As mentioned before, we need to induce the polarization of $A$ on the image of  the endomorphism $f$. Although this was also discussed in \cite[Prop. 3.9]{mascot}, they take a slightly different approach. Since we have explicitly constructed the lattice of the subvariety corresponding to the image of $f$ (see Remark \ref{R:basechica}), there is no need to find it inside the lattice of $A$, as done in \cite[Section 2]{mascot} (see their Remark 2.3).

For completeness and to fix notation, since it is the starting point of what we present in section \ref{S:matrizriemannchica}, we recall here briefly the method given in \cite{lroarch, lrrMathZ}  to find the induced polarization and a symplectic basis of a subvariety $A_f$ of a pav $A$: $\rho_r(f)$ is the rational representation of the endomorphism $f$ and $J_E$ the matrix of the polarization on $A$,  both with respect to a symplectic basis $\beta$ of $L$; $\gamma$ is a basis for $L_f$ given in the form of  a $2g \times 2h$ integral matrix $P_f$, whose columns are the coordinates  of the elements in $\gamma$ in the basis $\beta$; that is, $P_f$ defines the embedding  $A_f \hookrightarrow A$. Then the type of the induced polarization $D_f$ on $A_f$ is obtained by computing the elementary divisors of $P^t_f \cdot J_E \cdot P_f$.
Finally, a symplectic basis $\beta_f$ for $L_f$ expressed in terms of coordinates with respect to $\beta$ is obtained by applying the Frobenius algorithm (\cite[VI.3.  Lemma 1]{l}) to $\gamma$. So $\beta_f$ is captured as a $2g\times 2h$  matrix of coordinates with respect to the symplectic basis $\beta$ of $L$.

This method is implemented as an algorithm in \cite{rrgithub}; will refer to it as \textbf{Algorithm 2.1} in this work; it corresponds to the function \cite[InducedPolarization]{rrgithub} in our code.

\begin{rem}\label{R:basechica}

Note that the sublattice $L_f$  of $L$ corresponding to the subvariety $A_f$ is given by the pure lattice generated by  $\Ima(\rho_r(f))$; that is, $L_f=(\langle \rho_r(f)\rangle_{\ZZ} \otimes \mathbb{Q})\cap L,$ where $\langle \rho_r(f) \rangle_{\ZZ}$ denotes the lattice generated by the columns of $\rho_r(f)$. 

When $f$ comes from an idempotent in $\QQ[G]$, we can obtain $\rho_r(f)$ from the rational (symplectic) representation $\rho_r(G)$ of $G$. See \cite{brr} in case the action of $G$ on the Jacobian variety comes from the action of $G$ on a curve; the general result follows from this.
\end{rem}

\section{First method: Period matrix for a $G$-stable abelian subvariety of an abelian variety with the action of a group $G$}\label{S:achica}

Our next results deals with finding the period matrix for a subvariety $B$ of $A$ with $G$-action, given a symplectic representation of the $G$-action on $A$.

Let $G$ be a group and $(A=V/L,J_E)$ be a pav with $G$-action,  
where $J_E=\left(\begin{array}{c c} 0& E\\ -E& 0\\ \end{array}\right)$, $E=\textup{diag}(d_1,d_2,\dots, d_g)$. Assume the rational representation of $G$ is known: 
$$
\rho_r(G) \leq \Sp^E(2g \times 2g, \mathbb{Z}),
$$
where
$$
\Sp^E(2g \times 2g, \mathbb{Z}) =\{ N \in M(2g \times 2g, \mathbb{Z}) : N^t \cdot J_E \cdot N = J_E\}.
$$

Consider a subvariety $B  = V_B/L_B \subset A$   where $G$ acts (by restriction of the action on $A$). For instance, this is the case if $B = A_f =\Ima(f)$ where $f$ is any of the central idempotents in $\mathbb{Q}[G]$. 

Then a symplectic basis $\beta_B$ for $L_B$ can be computed as in Remark \ref{R:basechica}, and therefore the type $(n_1 , \ldots, n_h)$ of the induced polarization on $B$ is known. Our next result provides an algorithm to compute the symplectic representation of the action of $G$ on $B$ and a period matrix for $B$.

\begin{thm}\label{T:achica}
Let $A$ be a polarized abelian variety of dimension $g$  with $G$ action, and let $\beta$ be a symplectic basis for $A$ with respect to which the polarization on $A$ has matrix
$ 
J_E = \left( \begin{array}{cc}
	0 & E \\
	-E & 0
\end{array} \right)
$.

Assume the rational representation of $G$ in the basis $\beta$ is given by $\rho_r : G \to \Sp^E(2g \times 2g, \mathbb{Z})$.

Let $B$ be a subvariety of $A$ of dimension $h$ to which the action of $G$ on $A$ restricts, and let $i_B : B \to A$ denote the natural inclusion. Denote by $\beta_B$ a symplectic basis  for $B$ for which the induced polarization on $B$ has the form $ 
J_D := \left( \begin{array}{cc}
	0 & D \\
	-D & 0
\end{array} \right)
$, with $D = \textup{diag}(n_1, n_2,\dots, n_h)$. 

Then the rational representation of $G$ in the basis $\beta_B$ is given by $\rho_{r,B} : G \to \Sp^D(2h \times 2h, \mathbb{Z})$, where, for every $g \in G$, $\rho_{r,B}(g)$ is the unique matrix in $\Sp^D(2h \times 2h, \mathbb{Z})$ satisfying
\begin{equation} \label{eq:chica}
\rho_r(g) \cdot \rho_r(i_B) = \rho_r(i_B) \cdot \rho_{r,B}(g) , 
\end{equation}
where $\rho_r(i_B)\in M(2g\times 2h,\mathbb{Z})$ is the rational representation of $i_B$ with respect to the bases $\beta_B$ and $\beta$. 

Furthermore, the period matrix for $B=V_B/L_B$ with respect to the bases $\beta_B = \{u_1 , \ldots , u_{2h}\}$ for $L_B$ and  $\alpha_B = \left\{ \frac{1}{n_1} u_1, \ldots , \frac{1}{n_h} u_h \right\}$ for $V_B$ is of the form
$$ 
\Pi_B =(D \, Z_B) ,
$$ 
where $Z_B \in \mathbb{H}_h$ satisfies  
\begin{equation} \label{eq:zchica}
Z_B \, \gamma \, D^{-1} \, Z_B +D \, \alpha \, D^{-1} Z_B  -Z_B \, \delta -D\, \mu  = 0
\end{equation}
for each $g \in G$, where $ \rho_{r,B}(g) = \left( \begin{array}{cc}
	\alpha & \mu \\\gamma & \delta
\end{array}
\right)$, with $\alpha$, $\mu$, $\gamma$ and $\delta$  integral $h \times h$ matrices.  

Also, the analytic representation of the action of $g$ in $G$ restricted to $B$ is given by
$$
\rho_{a,B}(g) = (D\, \alpha +Z_B \, \gamma)\, D^{-1}\, .
$$
\end{thm}

\begin{proof}
	Writing $B =V_B/L_B$, we observe that since the action of $G$ on $A$ restricts to $B$, $\rho_r(g)(L_B)= L_B$ and $\rho_a(g)(V_B) = V_B$ for each $g \in G$. Therefore, denoting by $g_B$ the automorphism of $B$ obtained by restricting $g$ to $B$, we clearly have
	$$
	g \circ i_B = i_B \circ g_B .
	$$ 
	
	Now \eqref{eq:chica} is the matrix translation of this last equality, where $ \rho_{r,B}(g)$ is the $2h \times 2h$ rational representation of $g_B$ with respect to the basis $\beta_B$. To verify that $\rho_{r,B}(g) \in \Sp^D(2h \times 2h, \mathbb{Z})$, observe that
	$$
	\rho_r(i_B)^t \cdot J_E \cdot \rho_r(i_B) = J_D \, ,
	$$
	and hence
	\begin{align*}
		\rho_{r,B}(g)^t \cdot J_D \cdot \rho_{r,B}(g) &= \rho_{r,B}(g)^t \cdot \left( \rho_r(i_B)^t \cdot J_E \cdot \rho_r(i_B) \right) \cdot  \rho_{r,B}(g) \\
		& = \rho_r(i_B)^t \cdot \rho_r(g)^t \cdot J_E \cdot \rho_r(g) \cdot \rho_r(i_B)\\
		& = J_D \, .
	\end{align*}
	
	It is clear that the period matrix for $B$ with respect to the bases   $\beta_B = \{u_1 , \ldots , u_{2h}\}$ for $L_B$ and  $\alpha_B =  \left\{\frac{1}{n_1} u_1, \ldots , \frac{1}{n_h} u_h \right\}$ for $V_B$ is of the form
	$$ 
	\Pi_B =(D \, Z_B) ,
	$$ 
	where $Z_B \in \mathbb{H}_h$. Since for each $g \in G$ its restriction $g_B$ to $B$ is an automorphism of the polarized abelian variety $(B,J_D)$, the period matrix $\Pi_B$ satisfies 
	\begin{equation} \label{eq:chica2}
		\rho_{a,B}(g) \, \Pi_B = \Pi_B \, \rho_{r,B}(g). 
	\end{equation}

	Writing  $ \rho_{r,B}(g) = \left( \begin{array}{cc}
		\alpha & \mu \\\gamma & \delta
	\end{array}
	\right)$, with $\alpha$, $\mu$, $\gamma$ and $\delta$ integral  $h \times h$ matrices, and comparing both sides of \eqref{eq:chica2}, we see that
	$$
	\rho_{a,B}(g) = (D\, \alpha +Z_B \, \gamma)\, D^{-1}.
	$$
	and that \eqref{eq:zchica} holds.
\end{proof}

Clearly, Theorem \ref{T:achica} leads to an algorithm, whose code can be found at \cite[ActionGSubvariety.mgm]{rrgithub}. It  includes first to restrict the action from the ambient pav $A$ to a $G$-invariant subvariety $B$, with the function ActionGSubvariety, and then to use this restricted representation of $G$ to find the fixed Riemann matrices by this action. This last part is an upgrade of our algorithm in [3], where we found the set of Riemann matrices of ppav of dimension $g$ fixed by the action of $G$ represented in $\Sp(2g, \mathbb{Z})$, now with the function MoebiusInvariantDZ in \cite[polyDZ.m]{rrgithub}.

We use this result to find the period matrices of the isotypical factors in Sections \ref{S:g101} and \ref{S:g11}.

\begin{rem}\label{R:Schottky}
	In some cases, a family $\mathcal{Z}_{\lambda}$ of fixed matrices under the action of a given group will be found, with $\lambda$ in a set of complex parameters. Each element of $\mathcal{Z}_{\lambda}$ corresponds to a pav that shares the same action with the one we start with; that is, admitting the same $\rho_r(G)$ action. To determine explicitly the parameters $\lambda$ corresponding to the precise Jacobian or pav or family under study is in general difficult, as this is closely related to the Schottky problem; sometimes this can be achieved by using some extra known geometrical properties of the given variety. This is certainly a complication  that cannot be avoided, but, as a compensation, our methods can  produce numerical approximations as well as algebraic numbers, depending on the geometry of the variety and the action. They are effective methods to find period matrices that work in many cases in the context we are interested in: completely decomposable Jacobian varieties, CM-varieties, and others.
\end{rem}

Of course the result on Theorem \ref{T:achica} applies to the computation of a period matrix for the ambient abelian variety $A$ itself, but in practice the algorithm may fail computationally for $A$ if the dimension of $A$ is large, and still work for a $G$-invariant subvariety of lower dimension; as is the case in our examples, in Sections \ref{S:g101} and \ref{S:g11}.

\medskip

\section{Second method: find a period matrix for the image of $f\in \End_{\QQ}(A)$, given $\Pi_A=(E\, Z)$}\label{S:matrizriemannchica}

In order to obtain the period matrix of the subvariety $A_f=\textup{Im}(f)$, in this section we extend the method in \cite{lroarch, lrrMathZ}, whose outputs are: 

- a symplectic basis  $\beta_f$ for the lattice $L_f$ of $A_f$, 

- the rational representation $P_f$ of the inclusion $i_f:A_f\to A$, and 

- the induced polarization $D_f$ in $A_f$, 

Recall that its input is a period matrix $\Pi_A=(E\, Z)$ for the ambient  polarized abelian variety $A$. See Remark \ref{R:basechica} and what we call Algorithm 2.1.

\begin{thm}\label{T:matrizriemannchica}
Let $(A=V/L,J_E)$ be a pav with period matrix $\Pi_A=(E\; Z)\in M(g\times 2g,\CC)$ in suitable bases $\alpha$ for $V$ and $\beta$ for $L$, where $J_E=\left(\begin{array}{c c} 0& E\\ -E& 0\\ \end{array}\right)$, $E=\textup{diag}(d_1,d_2,\dots, d_g)$, and $Z\in \mathbb{H}_g$ is the Riemann matrix of $A$. 

For $f\in \End_{\QQ}(A)$, consider the subvariety $A_f :=\Ima(f) =V_f/L_f$ of dimension $h$ and a symplectic basis $\beta_f$ for its lattice $L_f$. 

Denote by  $i_f : A_f \to A$ the natural inclusion, by $P_f:=\rho_r(i_f) \in M(2g\times 2h,\ZZ)$ the matrix  of the rational representation of  $i_f $ with respect to the symplectic bases $\beta_f$ and $\beta$, and the induced polarization on $A_f$ by $D = \textup{diag}(n_1, n_2,\dots, n_h)$, with $(n_1 , \ldots, n_h)$ its type. Then

\begin{enumerate}
\item  If the symplectic basis for $L_f$ is $\beta_f = \{u_1 , \ldots , u_{2h}\}$, then $\alpha_f = \left\{\frac{1}{n_1} u_1, \ldots , \frac{1}{n_h} u_h \right\}$ is a basis for the complex vector space $V_f$, and the matriz $\rho_a(i_f) \in M(g\times h,\mathbb{C})$ for the analytic representation of $i_f$ with respect to the bases $\alpha_f$ and $\alpha$ is given by
$$
\rho_a(i_f) = (E \ Z) \, \beta_{f_1} \, D^{-1} ,
$$   
where $\beta_{f_1}, \beta_{f_2} \in M(2g\times h,\ZZ)$  are the two matrices such that $\rho_r(i_f) = \left( \beta_{f_1} \ \beta_{f_2}\right)$.

\item The period matrix of $A_f$ with respect to the bases $\alpha_f$ and $\beta_f$ is given by $$
\Pi_{A_f}=(D\; W) ,
$$
where $W \in \mathbb{H}_h$ is the unique solution to 
$$
 (E \ Z) \, \beta_{f_1} \, D^{-1} \, W =  (E \ Z) \, \beta_{f_2} .
$$
\end{enumerate}

\end{thm}
\begin{proof}
	
	According to Remark \ref{R:basechica}, we can find a basis $\gamma$ for $L_f$ in terms of $\beta$, use it to determine the type of the polarization $D$ of $A_f$ obtained by restriction of the polarization $E$ of $A$ to $A_f$, and then apply  the Frobenius algorithm to $\gamma$ to obtain a symplectic basis $\beta_f =\{u_1 , \ldots , u_{2h}\}$ for $L_f$.
	
	It follows that the rational representation $\rho_r(i_f)$ of the inclusion map $i_f : A_f \to A$ with respect to the bases $\beta_f$ and $\beta$ is the matrix in $M(2g \times 2h,\mathbb{Z})$ whose $j$-th column is given by the coordinates of $u_j$ with respect to  $\beta$, for $1 \leq j \leq 2h$, and the equality
	\begin{equation} \label{eq:pol}
		\mbox{}^t\rho_r(i_f) \, J_E \, \rho_r(i_f)  =\left(\begin{array}{c c} 0& D\\ -D& 0\\ \end{array}\right)
	\end{equation}
	is the matrix translation of the equality $i_f^*(E)= \widehat{i_f} \circ \lambda_E \circ i_f$, where $\lambda_E : A \to \widehat{A}$ is the isogeny associated to $E$.
	
	Since $D = \textup{diag}(n_1, n_2,\dots, n_h)$, with $(n_1 , \ldots, n_h)$ the type of $A_f$, it follows that taking $\alpha_f = \left\{ \frac{1}{n_1} u_1, \ldots , \frac{1}{n_h} u_h \right\}$ we obtain a basis for the complex vector space $V_f$, and that the period matrix for $A_f$ with respect to the bases $\alpha_f$ and $\beta_f$ has the form $\Pi_{A_f} = (D\; W),$ where $W \in \mathbb{H}_h$.
	
	The Hurwitz relation \eqref{eq:Hurwitz} then implies that
	\begin{equation} \label{eq:Hurwiz1}
		\rho_a(i_f) \ (D \ W) = (E \ Z) \, \rho_r(i_f),
	\end{equation}
	where $\rho_a(i_f)$ is the $g \times h$ matrix of the analytic representation of $i_f$ with respect to the bases $\alpha_f$ and $\alpha$ for $V_f$ and $V$ respectively. 
	
	Writing  the matrix $\rho_r(i_f) = \left( \beta_{f_1} \ \beta_{f_2}\right)$, with   $\beta_{f_1}, \beta_{f_2} \in M(2g\times h,\ZZ)$, we see that \eqref{eq:Hurwiz1} is equivalent to
	$$
	\rho_a(i_f) \, D = (E \ Z) \, \beta_{f_1}  \ \ \textup{and} \ \ \rho_a(i_f) \, W = (E \ Z) \, \beta_{f_2},
	$$
	from where it follows  that $\rho_a(i_f) = (E \ Z) \, \beta_{f_1} \, D^{-1},$ and that $W$ may be found from 
	$$
	(E \ Z) \, \beta_{f_1} \, D^{-1} \, W =  (E \ Z) \, \beta_{f_2} 
	$$
	follows, since $\rho_a(i_f) = (E \ Z) \, \beta_{f_1} \, D^{-1}$ has maximal rank $h$.
\end{proof}

\begin{rem}\label{R:uso-tipico}

Theorem \ref{T:matrizriemannchica} leads to Algorithm \ref{T:matrizriemannchica}, which can be found with code at \cite[ActionGSubvariety.mgm]{rrgithub} (functions IsotypicalFactorsAll and Subvariety). We use it combining resources from Magma \cite{magma} and  Sagemath \cite{sage} in Section \ref{S:g11} to find the period matrices of the primitive factors.

\end{rem}

\medskip

We point out that, once the period matrices for a set of subvarieties (fully) decomposing $A$ and the rational representation of the decomposing isogeny are known, it is possible to recover the period matrix for $A$ from these data. This is the case for the isotypical or GAD decompositions of $A$, for instance. Nevertheless, it is a technical result that is not actually needed for the purposes of this work, which is to decompose $A$ into simple factors. So for the sake of the length of this article, we decided not to include it here. It will be reported in a forthcoming work.

%\end{rem}

\section{Beyond the group algebra decomposition}\label{S:Beyond}

Let $(A=V/L, \cL_0)$ be a polarized abelian variety with the action of a (finite) group $G$.  From the results in Sections \ref{S:matrizriemannchica} and \ref{S:achica}, we can compute the period matrices for the isotypical factors $A_j$ and the primitive factors $B_j$ decomposing $A$ as follows: 

To obtain the period matrices for the $A_j$: if the rational representation for the action of $G$ on $L$ is given, apply Theorem \ref{T:achica}; if the Riemann matrix for $A$ is known, apply Theorem \ref{T:matrizriemannchica}. Recall from Section \ref{S:matricesperiodos} that each $A_j$ is the image of an explicit central idempotent $e_j \in \mathbb{Q}[G] \subset \End_{\mathbb{Q}}(A)$.  

Once the period matrix for $A_j$ has been found, since for each  $B_j$ one can find $f_j \in \mathbb{Q}[G] \subset \End_{\mathbb{Q}}(A)$ whose image is $B_j$ (see \cite{cr}), the period matrix for $B_j$ may be found by applying Theorem  \ref{T:matrizriemannchica} to $B_j \subset A_j$.

As we mentioned in the Introduction, a natural question is the comparison between the Group Algebra decomposition \eqref{e:gad} of $A$ 
$$
A \sim B_1^{n_1} \times \ldots \times B_r^{n_r} 
$$
and its Poincar\'e decomposition in terms of simple factors
$$
A \sim C_1^{k_1} \times \ldots \times C_s^{m_s}.
$$

The factors in the first one satisfy $\Hom_G (B_i^{n_i}, B_j^{n_j}) =0$, whereas $\Hom(C_i,C_j)=0$ for $i\neq j$. It may well happen that two different $B_j$ are isogenous, or that a $B_j$ is non-simple, as we will see in the examples.

In the next subsection \ref{SS:conocido} we recall some known results about the relation between subvarieties, idempotents, and the Neron-Severi group of a polarized abelian variety $A$, including a criterion to decide whether $A$ is simple in terms of its period matrix from \cite{alrJPAA}. 
We omit details and proofs, and refer to \cite{alrJPAA} and \cite{bl} for details.

Then, in subsection \ref{SS:nosotras}, we present a new technique  to actually decompose a polarized abelian variety into its simple factors.  

In particular, this technique applies to the primitive factors $B_j$ in the GAD decomposition of a polarized abelian variety with the action of a (finite) group $G$, since we can compute their period matrices as described above, and then apply the method to effectively decompose $B_j$ if it is non simple.  In this way we effectively go beyond the information that the group action gives.

We point out that the results in \ref{SS:conocido} allow us to determine whether a pav is simple or not, by a necessary and sufficient criterion. We worked out throughout the details and developed the effective method to decompose we present here. It corresponds to actually computing the period matrix of the subvariety if the criterion says it exists.

\subsection{Known results about subvarieties}\label{SS:conocido}

Let $(A=V/L, \cL_0)$ be a polarized abelian variety of type $(d_1,\dots, d_g)$, and consider  a symplectic basis $\{\lambda_1 , \ldots, \lambda_{2g}\}$ for $L$ and a basis for $V$ such that with respect to these bases the period matrix of $A$ is $(E \; Z )$, where $E=\text{diag}(d_1,\dots,d_g)$.  

If $x_1,\ldots,x_{2g}$ are the real 
coordinate functions of $L \otimes \RR$ associated to the given basis of $L$ and $z_1, \ldots , z_g$ are the 
complex  coordinate functions with respect to the given basis of $V$, these functions are related by the equation
\begin{equation} \label{eq3.1}
	\begin{pmatrix}z_1\\\vdots\\ z_g\end{pmatrix}=(E\;Z)\begin{pmatrix}x_1\\\vdots\\ 
		x_{2g}\end{pmatrix} \, .
\end{equation}

Considering $\{ dx_i \wedge dx_j \;|\; 1 \leq i < j \leq 2g \}$ as the canonical basis of $H^2(A,\QQ) = \wedge^2 \QQ^{2g}$,
$NS_{\QQ}(A)$ can be identified with 
\begin{equation} \label{eq3.2}
	NS_{\QQ}(A) = \{\omega\in\wedge^2\QQ^{2g}:\omega\wedge dz_1\wedge\cdots\wedge dz_g=0\},
\end{equation}
given by the image of the map

\begin{equation} \label{gamma}
	\begin{array}{c c c c}\gamma: & NS_{\QQ}(A) & \to &  H^2(A,\QQ) \\ & \mu & \mapsto & -\sum_{i < j}\mu(\lambda_i,\lambda_j) dx_i \wedge dx_j\\ \end{array}
\end{equation}

We also recall from \cite[Proposition 5.2.1]{bl} the  following  isomorphism  of $\QQ$-vector spaces  

\begin{equation}\label{varphi}
			\varphi : NS_{\QQ}(A)\to \End_{\QQ}(A)^s 
		\end{equation}
		defined by $\varphi(\cL)=\phi_{\cL_0}^{-1}\phi_{\cL}$ for $\cL \in NS_{\QQ}(A) $, where $\phi_{\cL} :A\to \widehat{A}$ is the isogeny induced by $\cL$.

\medskip

In \cite[Theorem 4.1]{alrJPAA} a necessary and sufficient criterion for the simplicity of $A$ in terms of its period matrix is given.  Using the above identifications, the criterion is translated into the existence of a tuple of rationals satisfying some nonlinear equations. In \cite{alrJPAA}, the corresponding equations for dimensions two and three are derived. Just for the sake of completeness we include here the equations for $A$ of dimension two. A similar system of non-linear equations arises in the higher dimensional situation.

\begin{cor}\label{p3.3}\cite[Prop. 4.4]{alrJPAA}
Let $(A,\mathcal{L})$ be a polarized abelian surface with period matrix 
$Z=\left(\begin{array}{cccc} 1& 0& z_{11} & z_{12} \\
	                     0& d&  z_{12} & z_{22}
                       \end{array}\right)$. Then $A$ admits a 
sub-elliptic curve if and only if there exists a vector 
$(a_{12},a_{13},a_{14},a_{23},a_{24},a_{34}) \in \QQ^6$ satisfying
\begin{eqnarray*}
-d&=&da_{13}+a_{24},\\
 0&=&(z_{11}z_{22}-z_{12}^2)a_{12} - da_{14}z_{11} + da_{13}z_{12} - a_{24}z_{12} + a_{23}z_{22} + da_{34} \; and \\
0&=&a_{14}a_{23}-a_{13}a_{24}+a_{12}a_{34}.\\
\end{eqnarray*}
\end{cor}

\subsection{Decomposing $A$ given its period matrix $\Pi = (E\, Z)$}\label{SS:nosotras}

As announced in the introduction (Section \ref{S:Intro}), in this work we push forward this result and, by pursuing the identifications in the previous subsection, we actually find the subvariety corresponding to the tuple solving the equations, so that we can explicitly decompose the variety $A$ in this way. In fact, we find the following procedure to decompose a pav $A$ given its period matrix; it actually describes simple subvarieties decomposing it and the corresponding isogeny, without computing the full endomorphism algebra.

\begin{thm}\label{beyond}
Let $A$ be a polarized abelian variety of dimension $g$ with period matrix $\Pi = (E\, Z)$. The following procedure yields a decomposition of $A$ into simple factors.

\begin{enumerate}
\item  For a given $n \in \{1 , \ldots , [\frac{g+1}{2}]\}$, look for $\omega=\sum_{i < j}a_{ij}dx_i \wedge dx_j$, with $a_{ij}\in \QQ$ satisfying all the equations in Theorem 4.1 in \cite{alrJPAA} (such as those in Corollary \ref{p3.3} for $g=2$). 
\item If such an $\omega$ exists, find $E_{\omega}=\gamma^{-1}(\omega)\in NS_{\QQ}(A)$ from  \eqref{gamma} and continue with (3) below. Otherwise, try with a different $n$. If there is no such $w$ for all $1\leq n \leq  [\frac{g+1}{2}]$, then $A$ is simple and we are done.
\item Find the symmetric idempotents $f_{\omega}=\varphi(E_{\omega})$ and $1-f_{\omega}$ in $\End_{\QQ}(A)^s$ described in \eqref{varphi}.
\item Find symplectic bases for the lattices of the subvarieties $A_{\omega}:=\Ima(f_{\omega})$ and $A_{\omega}^c:=\Ima(1-f_{\omega})$, and their induced polarizations, using Algorithm 2.1.
\item Find the period matrices for $A_{\omega}$ and $A_{\omega}^c$ using Theorem \ref{T:matrizriemannchica}.
\item Repeat the procedure for these subvarieties, using the corresponding period matrices obtained in the previous step, until all the simple factors have been found.
\end{enumerate}
\end{thm}

\begin{proof}
Steps $(1),(2),(3)$ are straightforward from the theory exposed earlier. Step $(4)$ gives complementary subvarieties of $A$, according to \cite[Prop. 2.3]{lrrMathZ}.
Using Algorithm \ref{R:basechica}, one obtains  bases for both subvarieties. Since the period matrix for $A$ is given, using the coordinates of the bases in $(4)$ one computes the period matrices for these two subvarieties. Finally, this procedure stops because $A$ is of finite dimension.
\end{proof}

This Theorem also leads to an algorithm, which is included in the Appendix as Algorithm \ref{alg:beyond}. It allows us to decompose varieties without the knowledge of its endomorphism algebra, and without  considering a group action on them, or even without having one, provided its period matrix is known. We use it in the proof of Corollary \ref{P:coro} stated in the introduction, which illustrates how to apply our methods, see  Sections \ref{S:g101} and \ref{S:g11}.

\section{Application 1: A genus $101$ curve with completely decomposable Jacobian variety} \label{S:g101}

In this section we prove the first statement of Corollary \ref{P:coro} presented in the Introduction.

In \cite{PA}, there is an example of a curve $X$ of genus $101$ such that the GAD for its Jacobian variety $JX$ has the form $S\times E_1 \times E_2^2\times E_3^8\times \dots \times E_{14}^8 \to JX,$ 
where $E_1,  \ldots, E_{14}$ are elliptic curves and $S$ is an abelian surface. Since $101$ is an  Ekedahl-Serre gap, it is of interest to find out if $S$ decomposes further.

We apply the results  in this work to find a period matrix for $S$, and show that $S$ indeed decomposes further. Hence, by going beyond GAD, we show that $JX$ is completely decomposable.

\subsection{A Riemann matrix for $S$.}

Consider the group $G:= \langle a , b\rangle$, where
\begin{align*}
a& :=(1,16,6,11)(2,18,8,15,5,19,9,12)(3,20,10,14,4,17,7,13)  \,  , \\
b &:=  (1, 20)(2, 19, 4, 17, 5, 16, 3, 18)(6, 15)(7, 14, 9, 12, 10, 11, 8, 13). 
\end{align*}

Then $G$ is the group labeled as $(800,980)$ in the SmallGroup Database of \cite{magma}, and it acts on a curve  $X$ of genus $101$  with signature $(0;8,8,2)$ and monodromy $(a,b,ab)$.  We use the algorithm from [3] to find the symplectic representation  $\rho_r(G)$ of $G$ associated to this action; it is stored in \cite[Grupo800-980.mgm]{rrgithub}.

Using \cite{cr}, we identify that $S$ is isogenous to the Jacobian variety of $X/H$  for $H$ the unique (up to conjugacy) abelian subgroup of order $100$ of $G$. Therefore, $S$ corresponds to the image of $JX$ under the idempotent
$$p_H=\frac{1}{|H|}\sum_{h\in H}\rho_r(h).$$

We use Algorithm 2.1 to describe the embedding  $i_{p_H}:S\to JX$, and the induced polarization on $S$. Thus we obtain a symplectic basis $\beta_H$ of $S$ in the coordinates of the symplectic basis of $JX$ in which $\rho_r(G)$ is given; that is,   we have a matrix $\rho_r(i_{p_H})$ in $M_{4\times 202}(\mathbb{Z})$. Since  the induced polarization on $p_H(S)$ is of type $(10,10)$,  it is a ppav.

Now, we follow Algorithm 4.1 to find the rational representation $\rho_{r,S}$ of the restricted action of $G$ on $S$. Since $\rho_r(a),\rho_r(b) \in \Sp(202,\mathbb{Z})$, $\rho_{r,S}(a)$ and $\rho_{r,S}(b)$ are found by solving the linear systems

$$\rho_r(a)\cdot \rho_r(i_{p_H})  = \rho_r(i_{p_H}) \cdot \rho_{r,S}(a),\text{ and } \rho_r(b)\cdot \rho_r(i_{p_H}) = \rho_r(i_{p_H}) \cdot \rho_{r,S}(b).$$

We obtain

$$\rho_{r,S}(a)=\left(\begin{array}{r r r r} 0&0&1&1\\ 1&-1&-1&1\\ -1&0&1&0\\ 1&-1&-1&0\\ \end{array}\right)^t$$
and
$$\rho_{r,S}(b)=\left(\begin{array}{r r r r} -1&1&1&-1\\ 0&0&1&1\\ -1&1&0&-1\\ 0&-1&0&1\\ \end{array}\right)^t  \, ,$$
and  check that $\rho_{r,S}(a), \rho_{r,S}(b) \in \Sp(4,\mathbb{Z})$.  The Riemann matrix $Z_S\in \mathbb{H}_2$ fixed by these matrices is given by
$$Z_S=\left(\begin{array}{r r} \frac{1+i\sqrt{2}}{2} & -\frac{1}{2}\\ -\frac{1}{2} & \frac{1+i\sqrt{2}}{2} \\ \end{array}\right).$$

\subsection{Elliptic curves on $S$}

Since $S$ is a ppav, we consider the following period matrix for the abelian surface $S$

$$\Pi_S=\left(\begin{array}{c c c c} 1& 0 &  \frac{1+i\sqrt{2}}{2} & -\frac{1}{2}\\ 0& 1& -\frac{1}{2} & \frac{1+i\sqrt{2}}{2} \\ \end{array}\right),$$
and use Algorithm \ref{alg:beyond} to decompose $S$ further. For this period matrix we use Corollary \ref{p3.3}. Hence we look for a vector $(a_{12},a_{13},a_{14},a_{23},a_{24},a_{34}) \in \QQ^6$
satisfying

\begin{eqnarray*}
-1&=&a_{13}+a_{24},\\
0&=&\left(\frac{-1+i\sqrt{2}}{2}\right)a_{12}-\left(\frac{1+i\sqrt{2}}{2}\right)a_{14}-\frac{a_{13}}{2}+\frac{a_{24}}{2}+\left(\frac{1+i\sqrt{2}}{2}\right)a_{23}+a_{34} \\
0&=&a_{14}a_{23}-a_{13}a_{24}+a_{12}a_{34}.\\
\end{eqnarray*}

One solution is $(a_{12}=\frac{1}{2},a_{13}=-\frac{1}{2},a_{14}=\frac{1}{2},a_{23}=0,a_{24}=-\frac{1}{2},a_{34}=\frac{1}{2})$.

It corresponds to the form 

$$\omega=\frac{1}{2} dx_1\wedge dx_2 -\frac{1}{2} dx_1\wedge dx_3 + \frac{1}{2} dx_1\wedge dx_4-\frac{1}{2} dx_2\wedge dx_4+\frac{1}{2} dx_3\wedge dx_4.$$

The corresponding element in $NS_{\QQ}(S)$ is

$$E_{\omega}=\frac{1}{2}\left(\begin{array}{r r r r} 0&-1&1&-1\\ 1&0&0&1\\-1&0&0&-1\\1&-1&1&0\\ \end{array}\right).$$

The period matrix $\Pi_S$ of $S$ is given in a symplectic basis, therefore its polarization is given by the matrix

$$E_0=\left(\begin{array}{c c c c} 0&0&1&0\\ 0&0&0&1\\-1&0&0&0\\0&-1&0&0\\ \end{array}\right).$$
The corresponding idempotent is $f_{\omega}=E_0^{-1}E_{\omega}$, and its complement is $1-f_{\omega}$. We obtain the following idempotents
 
$$f_{\omega}=\frac{1}{2}\left(\begin{array}{r r r r}  1&0&0&1\\ -1&1&-1&0\\ 0&-1&1&-1\\  1&0&0&1\\ \end{array}\right).$$

$$1-f_{\omega}=\frac{1}{2}\left(\begin{array}{r r r r}  1&0&0&-1\\ 1&1&1&0\\0&1&1&1\\ -1&0&0&1\\ \end{array}\right).$$

Denote by $L_{\omega}$ and $L_{\omega}^C$ the lattice of $\Ima(f_{\omega})$ and its complement, respectively. $L_{\omega}\otimes \QQ$ is the pure lattice generated by the columns of 
$f_{\omega}$. Therefore we have as basis $\{u_1=(1,-1,0,1), u_2=(0,-1,1,0)\}$. Analogously, for $L_{\omega}^C$ we obtain $\{v_1=(1,1,0,-1), v_2=(0,1,1,0)\}$.

To obtain the period matrices of $\Ima(f_{\omega})$ and its complement, we need to translate from coordinates to elements of the lattice. For this we multiply $\Pi_S \alpha$ for $\alpha$ in the corresponding basis.

For instance for $\Ima(f_{\omega})$ we have

$$u_1=\left(\begin{array}{c c c c} 1& 0 &  \frac{1+i\sqrt{2}}{2} & -\frac{1}{2}\\ 0& 1& -\frac{1}{2} & \frac{1+i\sqrt{2}}{2} \\ \end{array}\right)\left(\begin{array}{r} 1\\ -1\\0\\1\\ \end{array}\right)=\left(\begin{array}{c} \frac{1}{2}\\ \frac{-1+i\sqrt{2}}{2}\\ \end{array}\right),$$
and 
$$u_2=\left(\begin{array}{c c c c} 1& 0 &  \frac{1+i\sqrt{2}}{2} & -\frac{1}{2}\\ 0& 1& -\frac{1}{2} & \frac{1+i\sqrt{2}}{2} \\ \end{array}\right)\left(\begin{array}{r} 0\\-1\\1\\0\\ \end{array}\right)=\left(\begin{array}{c}\frac{1+i\sqrt{2}}{2}\\ -\frac{3}{2}\\ \end{array}\right).$$
Since $(1+i\sqrt{2})u_1=u_2$, we have that $\Ima(f_{\omega})$ is the elliptic curve with lattice generated by $\{1,1+i\sqrt{2}\}$. Similarly, its  complementary abelian subvariety is the elliptic curve with lattice generated by  $\{1,\frac{1+i\sqrt{2}}{3}\}$. Therefore, we have the sum isogeny

$$s: E_{1+i\sqrt{2}}\times E_{\frac{1+i\sqrt{2}}{3}}\to S.$$

The matrix $P=(u_1,v_1,u_2,v_2)$ corresponds to the rational representation of $s$, which is  

$$P=\left(\begin{array}{r r r r} 1&1&0&0\\ -1& 1&-1&1\\ 0&0&1&1\\ 1&-1&0&0\\ \end{array}\right);$$
it has determinant $4$, which corresponds to the degree of $s$.

The Hurwitz's equation satisfied by $s$ is

$$\left(\begin{array}{c c} \frac{1}{2} & \frac{3}{2} \\  \frac{-1+i\sqrt{2}}{2} &  \frac{1-i\sqrt{2}}{2}\\ \end{array}\right) \left(\begin{array}{c c c c} 1& 0 & 1+i\sqrt{2} & 0\\ 0& 1& 0& \frac{1+i\sqrt{2}}{3} \\ \end{array}\right)=\left(\begin{array}{c c c c} 1& 0 &  \frac{1+i\sqrt{2}}{2} & -\frac{1}{2}\\ 0& 1& -\frac{1}{2} & \frac{1+i\sqrt{2}}{2} \\ \end{array}\right)P,$$
where the matrix corresponding to the analytic representation of $s$
is
$$\left(\begin{array}{c c}  \frac{1}{2} & \frac{3}{2} \\  \frac{-1+i\sqrt{2}}{2} &  \frac{1-i\sqrt{2}}{2}\\ \end{array}\right).$$

Summarizing, there is an isogeny

$$E_{1+i\sqrt{2}}\times E_{\frac{1+i\sqrt{2}}{3}}\times E_1\times  E_2^2\times E_3^8\times \dots \times E_{14}^8 \to JX,$$
finding in this way a completely decomposable Jacobian variety of dimension $101$ and filling up this Ekedahl-Serre gap.

\section{Application 2: A completely decomposable  Jacobian variety of dimension $11$}\label{S:g11}
 
In this section we prove the second claim in Corollary \ref{P:coro}. This is, we exhibit here a new example of a completely decomposable Jacobian of a curve of genus $11$  finding its Riemann matrix explicitly, hence proving it is of CM-type since each elliptic curves in its decomposition has complex multiplication. 

Let $G = \langle a, b\rangle$ be the group labeled as $(96,28)$ in the SmallGroup Database of \cite{magma}
with

\begin{tabular}{ll}
a &:=(1,48,23,28,6,44,19,27,2,46,24,29,4,45,20,25,3,47,22,30,5,43,2 1,26)\\
&(7,42,17,34,12,38,13,33,8,40,18,35,10,39,14,31,9,41,16,36,11,37,15,32),\\ \\
b& :=(1,34,10,25)(2,36,11,27)(3,35,12,26)(4,31,7,28)(5,33,8,30) (6,32,9,29)\\
& (13,43,22,40)(14,45,23,42)(15,44,24,41)(16,46,19,37)(17,48,20,39) (18,47,21,38).\\
\end{tabular}

\medskip

It acts on a curve $X$ of genus $11$ with signature $(0; 24,4,2)$ and monodromy $(a,b,ab)$. We use the algorithm in \cite{brr} to obtain the associated rational representation $\rho_r : G\to \Sp(22,\ZZ)$; it is stored in \cite[Grupo98-28.mgm]{rrgithub}. 

However, a direct application of Theorem  \ref{T:achica} to compute the Riemann matrix $Z \in \mathbb{H}_{11}$  for $JX$ by finding the fixed matrix under the action of $\rho_r(G)$ fails computationally; so we take the approach of computing the period matrices of its decomposition into simple factors.

The isotypical decomposition of the Jacobian variety $JX$ corresponds to the following (sum) isogeny :
\begin{equation}\label{GAD_ex}
s:  A_1 \times A_2 \times A_3 \times A_4 \times A_5\to JX,
\end{equation}
where the $A_j$ are the isotypical factors. Using Theorem \ref{T:achica}, actually the corresponding Algorithm coded in Magma in [32], we find period matrices for them.

\begin{equation}\label{isotipicosg11}
\begin{split}
\Pi_{A_1}& =\left(\begin{array}{r | r} 6 & 6i\end{array}\right),\\
\Pi_{A_2}& =\left(\begin{array}{r r | r r} 4 & 0 & 2i\sqrt{3} & -2\\  0& 4 & -2& 2i\sqrt{3}\\  \end{array}\right),\\
\Pi_{A_3}& =\left(\begin{array}{r r | r r} 4 & 0 & 8i & -12i \\  0& 12 & -12i &  24i \\  \end{array}\right),\\
\Pi_{A_4}& =\left(\begin{array}{r r | r r} 3 & 0 & \frac{3i\sqrt{2}}{2} & -3\\  0& 6 & -3 & 3i\sqrt{2}\\  \end{array}\right),\\
\Pi_{A_5}& =\left(\begin{array}{r r  r r | r r r r} 2& 0 &0&0& 2i\sqrt{6} & 0& 3 i\sqrt{6}  & - i\sqrt{6} \\  0&2 &0&0& 0 & 2i\sqrt{6}&  i\sqrt{6}  & -3 i\sqrt{6} \\ 
0&0 &6&0&  3i\sqrt{6} & i\sqrt{6}& 6 i\sqrt{6}  & -3 i\sqrt{6} \\ 0&0 &0&6&  -i\sqrt{6} & -3 i\sqrt{6}& -3 i\sqrt{6}  & 6 i\sqrt{6} \\ 
\end{array}\right).
\end{split}
\end{equation}

See a further description of the rational representation $\rho_r(s)$ of the isogeny $s$ in Remark \ref{R:GAD_ex}.

Moreover, since the monodromy of this action is known, by \cite{yoibero} we obtain that each isotypical factor decomposes further as 
$A_1\sim E_1, A_2\sim E_2^2, A_3\sim E_3^2, A_4\sim E_4^2$ and $ A_5\sim S^2$, with $E_j$ elliptic curves and $S$ an abelian surfce. Therefore a GAD for $JX$ is 

\begin{equation}\label{GAD_2}
 E_1\times E_2^2 \times E_3^2 \times E_4^2 \times S^2\to JX,
 \end{equation}

The geometry of this action allows us to say that every $E_j$ in this GAD is isogenous to a Jacobian variety $J(X/H_j)$ of some intermediate curve $X/H_j$ for specific $H_j\leq G$. So we use Theorem \ref{T:matrizriemannchica} (with $E_j =\Ima( p_{H_j})$) finding

$$\Pi_{E_1}=\left(\begin{array}{r | r} 6 & 6i\end{array}\right),$$

$$\Pi_{E_2}=\left(\begin{array}{r | r} 8 & 4+ 4i\sqrt{3}\end{array}\right),$$

$$\Pi_{E_3}=\left(\begin{array}{r | r} 8 & 8i\end{array}\right),$$

$$\Pi_{E_4}=\left(\begin{array}{r | r} 3 & \frac{3i\sqrt{2}}{2}\end{array}\right),$$

For $S$, we use that there is a subgroup $K\leq G$ such that $J(X/K)\sim E_4\times S$, hence $S$ corresponds to the image of the idempotent $f_S:=p_Ke_{5}$ where $e_{5}$ is the central idempotent corresponding to the isotypical factor $A_5$. We use  Theorem \ref{T:matrizriemannchica} applied to $f_S$ and find

$$\Pi_S=4\left(\begin{array}{c c c c} 1& 0 &\frac{3i\sqrt{6}}{2}& 2i\sqrt{6}\\ 0& 3& 2i\sqrt{6}& 3i\sqrt{6}\\ \end{array}\right).$$

We then apply Corollary \ref{p3.3} and  look for $(a_{12},a_{13},a_{14},a_{23},a_{24},a_{34}) \in \QQ^6$ such that

\begin{eqnarray*}
 0&=&-48a_{12} - \frac{9i\sqrt{6}}{2}a_{14} +2i\sqrt{6}(3a_{13} -a_{24}) + 3i\sqrt{6}a_{23} + 3a_{34}, \\
-3&=&3a_{13}+a_{24}, \; \textup{and} \\
0&=&a_{14}a_{23}-a_{13}a_{24}+a_{12}a_{34}.\\
\end{eqnarray*}

One solution is $(a_{12}=0,a_{13}=-1,a_{14}=\frac{-4}{3},a_{23}=0,a_{24}=0,a_{34}=0)$, which corresponds to $\omega=-dx_1\wedge dx_3-\frac{4}{3}dx_1\wedge dx_4$.
Now we follow the steps on Theorem \ref{beyond} to effectively decompose $S$ further,  as we did in Section \ref{S:g101}.

We obtain the period matrix of $f_{\omega}(S)$:  $( 1\, \,  \frac{i\sqrt{6}}{6})$, and hence the decomposition of $JX$ into simple factors is
\begin{equation}\label{GAD_ex2}
JX\sim E_i^3\times E_{\frac{i\sqrt{2}}{2}}^2\times E_{\frac{1+i\sqrt{3}}{2}}^2\times E_{\frac{i\sqrt{6}}{6}}^4 ;
\end{equation}
thus showing that $JX$ is of CM type. Notice that they are not isogenous elliptic curves, hence this is the Poincar\'e decomposition of $JX$.

Comparing \eqref{GAD_2} and \eqref{GAD_ex2}, we notice that the primitive factors decomposing the isotypical factors $A_1$ and $A_3$ turn out to be isogenous. As said, in the isotypical decomposition of $A$, $\textup{Hom}_G(A_i,A_j)=\{0\}$  but not necessarily $\textup{Hom}_G(A_i,A_j)=\{0\}$ for $i\neq j$. Besides, the primitive factor $S$ in the isotypical factor $A_5$ is not simple.

\begin{rem}\label{R:GAD_ex}
Finally, we point out two interesting facts about the isogeny $s$ on \eqref{GAD_ex}. First,  the determinant of $\rho_r(s)$ is the degree of the isogeny decomposition. In this case it is equal to $(3456)^2=11943936 = 2^{14} 3^6$.

Secondly,  $\rho_r(s)$ satisfies $\rho_r(s)^t \cdot J_E \cdot \rho_r(s)=J_{\textup{diag}}$, where $J_E$ corresponds to the principal polarization on $JX$
and $J_{\textup{diag}}$ collects all the induced polarizations on the isotypical factors in \eqref{GAD_ex} or \eqref{isotipicosg11}. So

$$J_{\textup{diag}}=\left(\begin{array}{r r} 0 & D\\ -D & 0\\ \end{array}\right),$$
with $D=\textup{diag}(6,4,4,4,12,3,6,2,2,6,6)$.
\end{rem}

\section{Appendix} \label{S:app}

In this section we outline the algorithm emerging from Theorem \ref{beyond}, which allows us to find the Poincar\'e decomposition of a pav. The code for this algorithm, and the others in this work, can be found in  \cite{rrgithub}. The reader can also find there  more precise explanations on how to actually implement them in Magma \cite{magma}, as well as the calculations for our applications.

\vskip12pt

\noindent
\framebox{\parbox{\boxwidth}{
\begin{customalg}{\ref{beyond}}\label{alg:beyond}
Decomposition of a pav $A$ into simple factors given its period matrix.\\ \hrule 

\vskip6pt 

{\bf Input:} The period matrix $\Pi = (E\, Z)$ of $A$, with $E=\textup{diag}(d_1,\dots, d_g)$ and $Z\in \HH_g$.

{\bf Output:} The period matrices of all the simple factors in the decomposition of $A$. 

{\bf Algorithm:} \begin{enumerate*}
\item  For a given $n \in \{1 , \ldots , [\frac{g+1}{2}]\}$, look for $\displaystyle \omega=\sum_{1 \leq i < j \leq 2g}a_{ij}dx_i \wedge dx_j, a_{ij}\in \QQ$, \\
\mbox{ }\hspace{1.1cm} satisfying all the conditions in Theorem 4.1 in \cite{alrJPAA}. 
\end{enumerate*}
\begin{enumerate}
\setcounter{enumi}{1}
\item If such an $\omega$ exists, find $E_{\omega}=\gamma^{-1}(\omega)\in NS_{\QQ}(A)$, see \eqref{gamma} and continue with (3) below. Otherwise, try with a different $n$. If there is no such $w$ for all $1\leq n \leq  [\frac{g+1}{2}]$, then $A$ is simple.
\item Find the symmetric idempotents $f_{\omega}=\varphi(E_{\omega})$ and $1-f_{\omega}$, see \eqref{varphi}.
\item Find symplectic bases for the lattices of the subvarieties $\Ima(f_{\omega})$ and $\Ima(1-f_{\omega})$, and their induced representations, using Algorithm 3.1.
\item Find the period matrices for $\Ima(f_{\omega})$ and $\Ima(1-f_{\omega})$ using Algorithm 4.1.
\item Repeat the algorithm for the period matrices obtained in the previous step until all the simple factors have been found.
\end{enumerate}
\end{customalg}
}}

\end{document}